\documentclass[11pt]{amsart}
\usepackage[T1]{fontenc}
\usepackage[all]{xy}
\usepackage{relsize}
\usepackage{amssymb}
\usepackage{amsthm}
\usepackage{hyperref}
\usepackage{anyfontsize}
\hypersetup{colorlinks=true,linkcolor=blue,citecolor=magenta}
\usepackage{amsmath}
\usepackage{amscd,enumitem}
\usepackage{verbatim}
\usepackage{eurosym}
\usepackage{float}
\usepackage{color}
\usepackage{dcolumn}
\usepackage[mathscr]{eucal}
\usepackage[all]{xy}
\usepackage{hyperref}
\usepackage{bbm}
\usepackage[textheight=8.8in, textwidth=6.6in]{geometry}
\usepackage{mathtools}
\usepackage{tikz}
\usepackage{ytableau}
\usepackage{graphicx}
\usepackage{amssymb}
\usepackage{tabularx}
\usepackage{multirow}

\newtheorem*{thm*}{Theorem}
\newtheorem*{conj*}{Conjecture}

\newtheorem{theorem}{Theorem}[section]
\newtheorem{lemma}{Lemma}[section]

\newtheorem*{rmk}{Remark}

\newcommand{\Z}{\mathbb{Z}}

\newcommand{\R}{\mathbb{R}}

\newcommand{\C}{\mathbb{C}}
\newcommand{\N}{\mathbb{N}}

\newcommand{\pl}{\mathrm{PL}}
\newcommand{\D}{\mathcal{D}}
\newcommand{\XX}{\mathcal{X}}
\newcommand{\YY}{\mathcal{Y}}
\newcommand{\ZZ}{\mathcal{Z}}

\newcommand{\uu}{\underline}

\numberwithin{equation}{section}

\begin{document}

\subjclass[2022]{11, 11C, 11P, 11P82, 11Y, 11B, 05}
\keywords{Plane partition, Higher Tur\'{a}n inequality, Jensen polynomial, Taylor expansion, Hankel determinant}

\title{HIGHER TUR\'{A}N INEQUALITIES FOR THE PLANE PARTITION FUNCTION}
\author{Badri Vishal Pandey}
\address{Department of Mathematics, University of Virginia, Charlottesville, VA 22904}
\email{bp3aq@virginia.edu}

\begin{abstract}
    Here we study the roots of the doubly infinite family of Jensen polynomials $J_{\pl}^{d,n}(x)$ associated to MacMahon's plane partition function $\pl(n)$. Recently, Ono, Pujahari, and Rolen \cite{ono2022turan} proved that $\pl(n)$ is log-concave for all $n\geq 12$, which is equivalent to the polynomials $J_{\pl}^{2,n}(x)$ having real roots. Moreover, they proved, for each $d\geq 2$, that the $J_{\pl}^{d,n}(x)$ have all real roots for sufficiently large $n$. Here we make their result effective. Namely, if $N_{\pl}(d)$ is the minimal integer such that $J_{\pl}^{d,n}(x)$ has all real roots for all $n\geq N_{\pl}(d)$, then we show that 
    $$N_{\pl}(d)\leq 279928\cdot d(d-1)\cdot \left(6 d^3\cdot (22.2)^{\frac{3(d-1)}{2}}\right)^{2d} e^{\frac{\Gamma(2d^2)}{(2\pi)^{2d+2}}}  .$$
Moreover, using the ideas that led to the above inequality, we explicitly prove that $N_{\pl}(3)=26, N_{\pl}(4)=46, N_{\pl}(5)=73, N_{\pl}(6)=102$ and $N_{\pl}(7)=136$.
\end{abstract}
\maketitle
\section{Introduction and Statement of Results}
Given a sequence $\alpha:\N\to \R$ and positive numbers $d$ and $n,$ the associated \emph{Jensen polynomial of degree $d$ and shift $n$} is defined as
\begin{equation}
	J_{\alpha}^{d,n}(x):=\sum_{j=0}^d \binom{d}{j}\alpha(n+j) x^j.
\end{equation}
Notice that in the case of degree $d=2$, we have that
\[
	J_{\alpha}^{2,n-1}(x)=\alpha(n-1)+2\alpha(n)x+\alpha(n+1)x^2,
\]
whose roots are 
\[
	\dfrac{-\alpha(n)\pm \sqrt{\alpha(n)^2-\alpha(n-1)\alpha(n+1)}}{\alpha(n+1)}.
\]
In particular, $\alpha$ is log-concave at $n$ if and only if the roots of $J_{\alpha}^{2,n-1}(x)$ are real.

More generally, a polynomial with real coefficients is called \emph{hyperbolic} if all of its zeros are real. And so hyperbolicity of degree 2 Jensen polynomials is equivalent to the log-concavity of the associated sequence. The significance of hyperbolicity for higher degrees was recognized by the works of Jensen and P\'{o}lya in connection to Riemann hypothesis. Building on the work of Jensen, P\'{o}lya \cite{polya1927} proved that the Riemann hypothesis is equivalent to the hyperbolicity of all Jensen polynomials for the Taylor coefficients of the Riemann Xi-function at $s=1/2$. 

Generalizing log-concavity, we have Tur\'{a}n inequalities, which are of significant interest in combinatorics. Just like the log-concavity is equivalent to the hyperbolicity of $J_{\alpha}^{2,n}(x)$, the higher Tur\'{a}n inequalities are equivalent to the hyperbolicity of higher degree Jensen polynomials. 

A \emph{partition} $\lambda$ of positive integer $n$ is a finite non-increasing sequence $\lambda = (\lambda_1\geq\lambda_2\geq\cdots\geq\lambda_r>0)$ such that $\sum \lambda_i=n$. The \emph{partition function} $p(n)$ counts the number of partitions of $n$. DeSalvo and Pak \cite{DeSalvo2015} showed that the sequence $\{p(n)\}$ is log-concave i.e. $J_p^{2,n}(x)$ is hyperbolic for all $n\geq 25$. Chen, Jia and Wang \cite{chen2019} proved that $J^{3,n}_p(x)$ is hyperbolic for $n \geq 94$, and conjectured, for every degree $d\geq 1$, that there is a minimal integer $N_p(d)$ such that $J^{d,n}_p(x)$ is hyperbolic for all $n\geq N_p(d)$. Griffin et. al. \cite{griffin2019} proved their conjecture by showing that Jensen polynomials associated to partition function of each degree are hyperbolic for all sufficiently large shift $n$. Larson and Wagner \cite{larson2019} proved an effective form of this theorem by giving a upper bound for $N_p(d)$. Namely, they showed that $N_p(d)\leq (3d)^{24d}(50d)^{3d^2}$. Extending beyond the work of Chen et. al., they also proved that $N_p(4)=206$ and $N_p(5)=381$.

The work above on $p(n)$ is one case of a wider body of problems related to the Tur\'{a}n inequalities in the theory of partitions. More recently, inspired by a conjecture of Heim, Neuhauser and Tr\"{o}ger \cite{heim2021}, Ono, Pujahari and Rolen \cite{ono2022turan} investigated the 2-dimensional analog of $p(n)$ which is the {\it plane partition function} (for background, see references by Andrews \cite{andrews_1984} and Stanley \cite{stanley1989}). A \emph{plane partition} of size $n$ is an array of non-negative integers $\pi:=(\pi_{i,j})$ such that $\sum_{i,j} \pi_{i,j}=n$, in which the rows and columns are non-increasing. Below is a 3-d rendering of a plane partition for $n=30.$
\begin{figure}[H]
	\begin{center}
    \begin{tikzpicture}[scale=0.7]
      \draw (0,0) -- (2,0) -- (2,1) -- (3,1) -- (3,2) -- (4,2) -- (4,3) -- (0,3) -- (0,0);
      \draw (1,0) -- (1,3);
      \draw (2,1) -- (2,3);
      \draw (3,2) -- (3,3);
      \draw (0,1) -- (2,1);
      \draw (0,2) -- (3,2);

      \node at (0.5,0.5) {\Large $3$};
      \node at (0.5,1.5) {\Large $4$};
      \node at (0.5,2.5) {\Large $5$};
      \node at (1.5,0.5) {\Large $1$};
      \node at (1.5,1.5) {\Large $3$};
      \node at (1.5,2.5) {\Large $5$};
      \node at (2.5,1.5) {\Large $2$};
      \node at (2.5,2.5) {\Large $4$};
      \node at (3.5,2.5) {\Large $3$};

      \draw[->, very thick] (5.5,1.5) -> (6.5,1.5);

      \node[inner sep=0pt] (boxes) at (10,1.5)
{\includegraphics[scale=0.5]{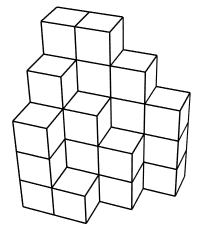}};
\end{tikzpicture}
\end{center}
\caption{\label{fig:partition-for-n=30}A plane partition for $n=30$}
\end{figure}
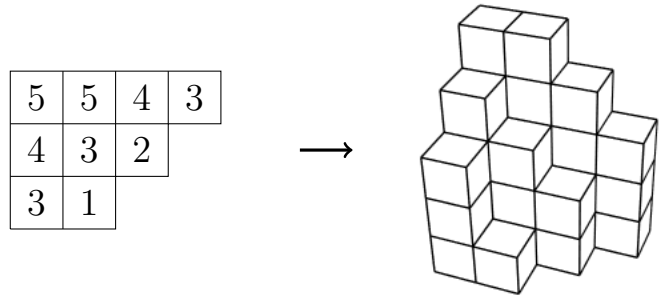

The plane partition function $\pl(n)$ counts the number of plane partitions of size $n$. MacMahon \cite{macmahon2004} proved that
\[
f(x)=\sum_{n=0}^{\infty}\pl(n) x^n:=\prod_{n=1}^{\infty}\dfrac{1}{(1-x^n)^n}=1+x+3x^2+6x^3+13x^4+24x^5+48x^6+\cdots.
\]
This function is of great importance in physics. It appears prominently in connection with the enumeration of small black holes in string theory, as $f(x)$ is the generating function (for example, see Appendix E of \cite{Dabholkar_2005}) for the number of BPS bound states between a $D6$ brane and $D0$ branes on $\C^3.$

Heim, Neuhauser and Tr\"{o}ger \cite{heim2021} undertook the study of the plane partitions in analogy with the hyperbolicity results of the partition function $p(n)$. They proved many inequalities satisfied by $\pl(n)$ including proving that $\pl(n)$ is log-concave for sufficiently large $n$. They also conjectured the bound to be $12$. Ono, Pujahari and Rolen \cite{ono2022turan} proved this conjecture. In addition, they also proved that for each degree $d$, the Jensen polynomials are hyperbolic for sufficiently large shift $n$. To prove this, they gave a strong asymptotic formula for the plane partition function derived using Wright's flexible circle method \cite{wright1931}. This asymptotic formula satisfies conditions required by Theorems 3 and 6 of \cite{griffin2019}, which implies that the limiting behavior of $J_{\pl}^{d,n}(x)$ as $n\to\infty$ can be modeled by \emph{Hermite polynomials} which are known to be hyperbolic. This proves their theorem. 

Here we make their result effective. More precisely, for any $d$, suppose that $N_{\pl}(d)$ is the minimal integer for which every Jensen polynomials of degree $d$ are hyperbolic for all shifts $n\geq N_{\pl}(d).$ Then we give an upper bound on $N_{\pl}(d)$.

\begin{theorem}\label{thm:main_theorem_general_case}
	For a positive integer $d\geq 4$, we have 
	$$ N_{\pl}(d)\leq 279928\cdot d(d-1)\cdot \left(6 d^3\cdot (22.2)^{\frac{3(d-1)}{2}}\right)^{2d} e^{\frac{\Gamma(2d^2)}{(2\pi)^{2d+2}}} .$$
\end{theorem}

Moreover, by working explicitly with the expression that arises in the proof of Theorem \ref{thm:main_theorem_general_case}, we are able to compute $N_{\pl}(d)$ for the cases $d=3,4,5,6$ and $7$.

\begin{theorem}\label{thm:main_theorem_special case}
	We have that $N_{\pl}(3)=26, N_{\pl}(4)=46, N_{\pl}(5)=73, N_{\pl}(6)=102$ and $N_{\pl}(7)=136$.
\end{theorem}


\begin{center}
\begin{tabular}{ |c|c|c|c|c|c|c|c|c|c|c|c|c|c|c| } 
\hline
$d$ & 8 & 9 & 10 & 11 & 12 & 13 & 14 & 15 & 16 & 17 & 18 & 19 & 20 \\
\hline 
$N_{\pl}(d)$ & 173 & 215 & 260 & 307  & 359 & 414 & 472 & 533 & 596 & 662 & 731 & 803 & 873\\  
\hline
\end{tabular}\\ \vspace{0.2cm}
Table:Conjectural value for $N_{\pl}(d)$ for small $d$.
\end{center}

\begin{rmk}
	It is quite surprising that $N_{\pl}(d)$ for smaller values of $d$ is smaller than corresponding values of $N_p(d)$ since plane partition function has much more complex asymptotic formula than partition function.
\end{rmk}

The overall strategy of the proofs is analogous to that employed by Larson and Wagner \cite{larson2019} in their work on the partition function. However, we emphasize that the calculations here are significantly more complicated because we do not enjoy the modularity of the generating function. As a consequence, obtaining asymptotics for plane partition function is well-known to be far more difficult than for $p(n)$.

Following Hermite, we have a sufficient condition for a polynomial to be hyperbolic in terms of positivity of certain naturally arising determinants. We will consider normalized Jensen polynomials $J_{\pl}^{d,n}(x)/\pl(n)$, in which case, these determinants turn out to be polynomials in $\pl(n+j)/\pl(n)$, where $0\leq j\leq d$. Thanks to Ono, Pujahari and Rolen \cite{ono2022turan}, we have an infinite family of strong asymptotic formulas for $\pl(n)$ with explicit bounds on the error terms. We make judicious choices of the parameters which help us obtain functions that closely approximate these ratios. Analyzing the bounds on the error terms effectively leads to Theorem \ref{thm:main_theorem_general_case}. For $d=3,4,5,6$, and $7$, direct computations give rise to good bounds, allowing us to reduce Theorem \ref{thm:main_theorem_special case} to a reasonably small finite number of cases.

This paper is organized in the following way. In Section \ref{sec:plane-partition}, we study strong asymptotic formulas for $\pl(n)$. In Section \ref{sec:hermite-condition}, we consider certain Hankel determinants and determine their implications for the limiting behavior of $J_{\pl}^{d,n}(x)$, as $n\to\infty$. We study asymptotics of the ratios of plane partition, and also obtain precise bounds on the errors. Finally, we prove Theorem \ref{thm:main_theorem_general_case} and \ref{thm:main_theorem_special case} using the approximation function for the ratios of plane partitions and error estimates accumulated.

\section*{acknowledgements}
I would like to thank Ken Ono for suggesting and guiding this project, and also  for providing research support with his NSF Grant DMS-2055118. I would like to thank Alejandro De Las Penas Castano, Hasan Saad and Wei-Lun Tsai for helpful discussions. I would also like to thank Alejandro De Las Penas Castano for helping me with Mathematica.

\section{Asymptotic formula of $\pl(n)$}\label{sec:plane-partition}
In \cite{ono2022turan}, Ono, Pujahari and Rolen obtained very strong asymptotic formulas for $\pl(n)$. In fact, there are infinitely many formulas, one for each positive integer $r$, where for large n, the implied error terms are smaller with larger choice of $r$. To make this precise, we need two constants
\begin{equation}
	A:=\zeta(3)\approx 1.20206...,\quad \text{and} \quad c:=2\int_0^{\infty}\dfrac{y\log{y}}{e^{2\pi y}-1}dy=\zeta'(-1)\approx -0.16542....
\end{equation} 
Furthermore, for non-negative integers $s$ and $m$, define coefficients $c_{s,m}(n)$ by 
\begin{equation}
	\sum_{n=0}^{\infty}c_{s,m}(n)y^n:=\dfrac{(1+y)^{2s+2m+\frac{13}{12}}}{(3+2y)^{m+\frac{1}{2}}}.
\end{equation}
We define an important parameter using these coefficients,
\begin{equation}
	b_{s,m}:=c_{s,m}(2m).
\end{equation}
The asymptotic formulas are given in terms of special numbers $\beta_0,\beta_1,...$ defined by 
\begin{equation}
	\sum_{n=0}^{\infty}\beta_s y^s := \exp\left(-\sum_{i=1}^{\infty} \alpha_i y^i\right),
\end{equation} 
where
\begin{equation}
	\alpha_s:=\dfrac{2\Gamma(2s+2)\zeta(2s)\zeta(2s+2)}{s(2\pi)^{4s+2}}.
\end{equation}
Also, to reduce the complexity of error terms, for non-negative $r$, Ono et. al. defined 
\begin{equation*}
	n_r:=\min\left\{n\geq 1\,:\, 0.056\cdot \sum_{s=1}^{r+1}\left(\dfrac{s\cdot A^{\frac{1}{3}}}{2^{\frac{7}{6}}n^{\frac{1}{3}}}\right)^{2s}\left(\dfrac{\pi^2n^{\frac{1}{3}}}{(2A)^{\frac{1}{3}}s}+2\right)<1\right\},
\end{equation*}
and
\begin{equation*}
	l_r:=\min\left\{n\geq 1\,:\,2^{r+4}\pi^3\alpha_{r+2}\left(\dfrac{2A}{n}\right)^{\frac{2r+4}{3}}+5e^{-4.7(\frac{n}{2A})^{1/3}}<\dfrac{1}{2}\right\}.
\end{equation*}
The explicit bounds on the error terms are given in terms of $\XX_r(n),\YY_r(n),\ZZ_r(n)$. To define $\XX_r(n)$, and $\YY_r(n)$, we let
\begin{equation}\label{eq:C_r}
	C_r:=2\max_{|z|=1}\left\{\left|e^{-\sum_{s=1}^{r+1} \alpha_s z^s}\right|\right\}.
\end{equation}
We require one additional parameter to define $\ZZ_r$. First, we define
\begin{equation}
	\chi_s(t):=\dfrac{v^{2s+\frac{25}{12}}\sqrt{2v+1}}{2\pi (v^2+v+1)},
\end{equation}
where $t^2=3-2v-v^{-2}$. Using this we define the parameter
\begin{equation}
	D_r:= \dfrac{1}{(2r+4)!}\max\left\{\max\left\{\left|\chi_s^{(2r+4)}(t)\right|\right\}_{t\in\R}\right\}_{s=0}^{r+1}.
\end{equation}
Now we define $\XX_r(n),\YY_r(n)$, and $\ZZ_r(n)$ by
\begin{equation}
	\XX_r(n):=e^{c+AN_n^2}2^{r+\frac{49}{24}}C_r N_n^{-2r-\frac{49}{12}},
\end{equation}
\begin{equation}
	\YY_r(n):=\left|e^{c+AN_n^2}\left(2^{r+5}\pi^3\alpha_{r+2}N_n^{-2r-4}+10e^{-4.7N_n}\right)\left(2^{r+\frac{49}{24}}C_rN_n^{-2r-\frac{49}{12}}+\sum_{s=0}^{r+1}2^{s+\frac{1}{24}}\beta_sN_n^{-2s-\frac{1}{12}}\right)\right|,
\end{equation}
and
\begin{equation}
	\ZZ_r(n):=e^c\left(D_r\cdot\Gamma\left(r+\dfrac{5}{2}\right)(AN_n^2)^{-r-\frac{5}{2}}e^{3AN_n^2}+0.64\cdot 2^{r+1}e^{2AN_n^2}\right)\sum_{s=0}^{r+1}\beta_sN_n^{-2s-\frac{13}{12}}.
\end{equation}
where $N_n:=\left(\frac{n}{2A}\right)^{1/3}$. With the notation above, Ono et. al. proved the following theorem.
\begin{theorem}[\cite{ono2022turan}, Theorem 1.3]
	If $r\in\Z^+$, then for every integer $n\geq \max(n_r,l_r,87)$, then we have that 
	\begin{equation}\label{formula:plane-partition}
	\pl(n)=\dfrac{e^{c+3AN_n^2}}{2\pi}\sum_{s=0}^{r+1}\sum_{m=0}^{r+1}\dfrac{(-1)^m 		\beta_s b_{s,m}\Gamma(m+\frac{1}{2})}{A^{m+\frac{1}{2}}N_n^{2s+2m+\frac{25}{12}}}+E_r^{{\rm{maj}}}(n)+E^{{\rm{min}}}(n),
	\end{equation}
	where
	\begin{equation}\label{eq:minor_error}
		|E^{\min}(n)|\leq \exp\left(\left(3A-\frac{2}{5}\right)N_n^2\right),
	\end{equation}
and 
	\begin{equation}\label{eq:major_error}
	|E_r^{\rm{maj}}(n)|\leq \dfrac{(\XX_{r}(n)+\YY_r(n))e^{2AN_n^2}}{N_n\pi}+|\ZZ_r(n)|.
	\end{equation}
\end{theorem}
This formula is proved using Wright's circle method. The quantity $E^{\rm{min}}(n)$ arises from minor arc integrals, and $E_r^{\rm{maj}}(n)$ arises from major arc integrals.

To obtain our results, we must make this result effective, and then make good choices of the parameters for our application. To this end, we make the following change of variables:
\begin{equation}\label{eq:formula-for-w-and-delta}
	w(n):=\dfrac{2^{1/3}}{\sqrt{3}A^{1/6}n^{1/3}}\quad \text{ and } \quad \delta(n):=\dfrac{\sqrt{3A}}{2}w(n)^{2}.
\end{equation}
For our purpose, we restrict to the case when $w\in [0,\varepsilon_{r,d}]$, where 
\begin{align}
	\varepsilon_{r,d} := & \, d^{-2d} 2^{-d(d-1)}\cdot \left(\frac{4e}{\sqrt{3A}}\right)^{-d(d-1)} \left( e^{\frac{\Gamma(2d^2)}{(2\pi)^{2d^2+2}}}6^{2d-2} (6A)^r (r+1) \right)^{-\frac{1}{3}} \left(0.1485\cdot 2^{14} (3A)^{3}\pi^3 \right)^{-\frac{1}{3}},
\end{align}
corresponding to our eventual bound on $N_{\pl}(d)$ for right choice of $r$ (depending on $d$), since we also want to give an upper bound on errors for $w\in[0,\varepsilon_{r,d}]$. 

\begin{theorem}\label{thm:formula-for-pl-in-w}
	If $r\in\Z^+$ and $w:=w(n)$, then for every $w\in [0,\varepsilon_{r,d}]$ we have 
	\begin{equation*}
		\pl(n)= \widehat{\pl}_r(w) + E_r(w),
	\end{equation*}
	where 
	\begin{align}
	\widehat{\pl}_r(w):=&\dfrac{e^{c+\frac{1}{w^2}}w^{\frac{25}{12}}}{2\pi}\sum_{s=0}^{r+1}\sum_{m=0}^{r+1}f_{s,m} w^{2s+2m}\\ \nonumber
	:=&\dfrac{e^{c+\frac{1}{w^2}}w^{\frac{25}{12}}}{2\pi}\sum_{s=0}^{r+1}\sum_{m=0}^{r+1}(-1)^m \beta_s b_{s,m}\Gamma\left(m+\frac{1}{2}\right)3^{s+m+\frac{25}{24}}A^{s+\frac{13}{24}} w^{2s+2m},
	\end{align}
	and 
	\begin{align}
		\left|E_r(w)\right| & \leq e^{c+\frac{1}{w^2}}\cdot C_r \cdot 2^{r+8+\frac{1}{24}} \pi^2  (3A)^{r+\frac{61}{24}} (r+2) \cdot w^{2r+5+\frac{1}{12}}.
	\end{align}
\end{theorem}
\begin{rmk}
	We stress that $\widehat{\pl}_r(w(n))\sim\pl(n)$ as $n\to \infty$.
\end{rmk}
\begin{proof}
First let's convert $\XX_r,\YY_r$ and $\ZZ_r$ from $n$ to $w$ using \eqref{eq:formula-for-w-and-delta}. 
\begin{align}
	\widehat{\XX}_{r}(w):=\XX_r\left(\dfrac{2}{3^{3/2}\sqrt{A}w^3}\right)& =e^{c+\frac{1}{3w^2}}2^{r+\frac{49}{24}}C_r (3A)^{r+\frac{49}{24}} w^{2r+4+\frac{1}{12}}\\ \nonumber
	& \leq e^{c+\frac{1}{3w^2}}\cdot C_r \cdot 2^{r+6+\frac{1}{24}} \pi^3  (3A)^{r+\frac{49}{24}} (r+2) \cdot w^{2r+4+\frac{1}{12}}.
\end{align}
where by \eqref{eq:C_r}, we have
\begin{align}
	C_r\leq 2\cdot e^{\max\{0.02,\frac{(r+1)}{2}\alpha_{r+1}\}} =	&\begin{cases}
				2\cdot e^{\frac{(r+1)}{2}\alpha_{r+1}} & \text{if } r\geq 22\\
				2 e^{0.02} & \text{otherwise},
			\end{cases}
\end{align}
since $\alpha_{r+1}\geq \alpha_{r}$ for all $r > 21$ and $10^{-3}<\alpha_r<10^{-17}$ for $r \leq 21$. We also have
\begin{align}
	\widehat{\YY}_{r}(w)&:=\YY_r\left(\dfrac{2}{3^{3/2}\sqrt{A}w^3}\right)\\ \nonumber
	&=\Bigg|e^{c+\frac{1}{3w^2}}\left(2^{r+5}\pi^3\alpha_{r+2} (3A)^{r+2} w^{2r+4}+10e^{-4.7 \frac{1}{\sqrt{3A}w}}\right) \\ \nonumber
	&\qquad \times\left(2^{r+\frac{49}{24}}C_r (3A)^{r+\frac{49}{24}}w^{2r+\frac{49}{12}}+\sum_{s=0}^{r+1}2^{s+\frac{1}{24}}\beta_s (3A)^{s+\frac{1}{24}} w^{2s+\frac{1}{12}}\right)\Bigg|\\ \nonumber
	&\leq 2\cdot 2\cdot e^{c+\frac{1}{3w^2}}2^{r+5+\frac{1}{24}} \pi^3  (3A)^{r+\frac{49}{24}} \alpha_{r+2} (r+2) w^{2r+4+\frac{1}{12}} \\ \nonumber
	& \leq e^{c+\frac{1}{3w^2}}\cdot C_r \cdot 2^{r+6+\frac{1}{24}} \pi^3  (3A)^{r+\frac{49}{24}} (r+2) \cdot w^{2r+4+\frac{1}{12}},
\end{align}
and
\begin{align}
	\widehat{\ZZ}_{r}(w)&:=\ZZ_r\left(\dfrac{2}{3^{3/2}\sqrt{A}w^3}\right)\\ \nonumber
	&=e^c\left(D_r\cdot\Gamma\left(r+\dfrac{5}{2}\right)(3w^2)^{r+\frac{5}{2}}e^{\frac{1}{w^2}}+0.64\cdot 2^{r+1}e^{\frac{2}{3w^2}}\right)\sum_{s=0}^{r+1}\beta_s (3A)^{s+\frac{13}{24}}w^{2s+\frac{13}{12}} \\ \nonumber
	& \leq e^{c+\frac{1}{w^2}}\cdot 2\cdot D_r\Gamma\left(r+\frac{5}{2}\right) 3^{r+\frac{5}{2}+\frac{13}{24}}A^{\frac{13}{24}}(r+2)w^{2r+5+\frac{13}{12}}.
\end{align}
We investigate $D_r$. To this end we recall that
\begin{align*}
	\chi_s(t)=\dfrac{v^{2s+\frac{25}{12}}\sqrt{2v+1}}{2\pi (v^2+v+1)},
\end{align*}
where $t^2=3-2v-v^{-2}$. First, since we have that 
\begin{equation}\label{eq:dv/dt}
	\dfrac{dv}{dt} = \dfrac{tv^3}{1-v^3}=i\dfrac{v^2\sqrt{2v+1}}{v^2+v+1},
\end{equation}
we get
\begin{align}
	\dfrac{d}{dt}v^k = & i\dfrac{kv^{k+1} \sqrt{2v+1}}{v^2+v+1}; \label{eq:dv} \\ 
	\dfrac{d}{dt}(2v+1)^{k} = & i\dfrac{(2k)(2v+1)^{k-1/2}v^2}{v^2+v+1}; \label{eq:2v+1} \\ 
	\dfrac{d}{dt}(v^2+v+1)^{-k} = & i\dfrac{(-k)v^2(2v+1)^{3/2}}{(v^2+v+1)^{k+2}}. \label{eq:v^2+v+1}
\end{align}
The parameterization of the curve traced by $v$ is given by $x\pm \sqrt{\sqrt{x}-x^2} , x\in[0,1]$ (see page 21 last paragraph of \cite{ono2022turan}). Using Mathematica, one checks that 
\begin{align*}
	|v|\leq 1;  \hspace{1cm}  1\leq \left| 2v+1 \right| \leq 3; \hspace{1cm} 0.9621 \leq \left| v^2+v+1 \right| \leq 3.
\end{align*}
One can differentiate $\chi_s(t)=fgh$ with respect to $t$ using the product rule (where $f=v^{k_1}, g=(2v+1)^{k_2}$ and $h=(v^2+v+1)^{k_3}$) and the composition rule and using \eqref{eq:dv/dt}-\eqref{eq:v^2+v+1}, a simple induction shows that the $n$-th derivative of $\chi_s(t)$ has $3^n$ terms of the form $c_{k_1,k_2,k_3}\cdot\frac{v^{k_1}(2v+1)^{k_2}}{(v^2+v+1)^{k_3}}$. The values of $c_{k_1,k_2,k_3}$ are bounded above by $\prod_{k=0}^{n}\left(2s+2+\frac{25}{12}+k\right)$, and maximizing the possible powers of each of the $v, 2v+1, v^2+v+1$, one gets that 
\begin{align*}
	\left| \dfrac{d^n}{dt^n} \chi_{s}(t) \right| \leq  \dfrac{3^n 3^{\frac{3n}{2}+\frac{1}{2}}}{2\pi (0.9621)^{2n}}\prod_{k=0}^{n}\left(2s+2+\dfrac{25}{12}+k\right).
\end{align*}
Hence, by the definition of $D_r$, we have that 
\begin{equation}
	D_r \leq   \dfrac{3^{5r+10+\frac{1}{2}}}{2\pi (2r+4)!(0.9621)^{4r+8}}\prod_{k=0}^{2r+4}\left(2r+2+\dfrac{25}{12}+k\right).
\end{equation}
So we get that 
\begin{align}
	\widehat{\ZZ}_{r}(w) & \leq e^{c+\frac{1}{w^2}}\cdot 2\cdot \dfrac{2^{2r+4}3^{5r+10+\frac{1}{2}}}{(0.9621)^{4r+8}}\Gamma\left(r+\frac{5}{2}\right) 3^{r+\frac{5}{2}+\frac{13}{24}}A^{\frac{13}{24}}(r+2)w^{2r+5+\frac{13}{12}} \\ \nonumber
						 & \leq e^{c+\frac{1}{w^2}}\cdot 2\cdot \dfrac{2^{2r+4}3^{5r+10+\frac{1}{2}}}{(0.9621)^{4r+8}}\Gamma\left(r+\frac{5}{2}\right) 3^{r+\frac{5}{2}+\frac{13}{24}}A^{\frac{13}{24}}(r+2)\cdot\varepsilon_{r,d}\cdot w^{2r+5+\frac{1}{12}} \\ \nonumber
	& \leq e^{c+\frac{1}{w^2}}\cdot C_r \cdot 2^{r+6+\frac{1}{24}} \pi^2  (3A)^{r+\frac{61}{24}} (r+2) \cdot w^{2r+5+\frac{1}{12}},
\end{align}
where the last inequality comes after substituting the value $\varepsilon_{r,d}$ and comparing with $\widehat{\XX}_r$. 
We also have
\begin{align*}
	\left| E^{\min}(n)\right| \leq & e^{c+\frac{1}{w^2}}\cdot C_r \cdot 2^{r+6+\frac{1}{24}} \pi^2  (3A)^{r+\frac{61}{24}} (r+2) \cdot w^{2r+5+\frac{1}{12}}.
\end{align*}
Everywhere above we are using the fact that $w$ is small enough so that the dominant term is the term with the smallest power of $w$. This gives us
\begin{align*}
	\left|E_r(w)\right| & \leq \left|E_r^{{\rm{maj}}}(n)+E^{{\rm{min}}}(n)\right|\\ \nonumber
			& \leq 	\left|\dfrac{(\widehat{\XX}_{r}(w)+\widehat{\YY}_r(w))e^{2AN_n^2}}{N_n\pi}\right|+\left|\widehat{\ZZ}_r(w)\right|+ \exp\left(\left(3A-\frac{2}{5}\right)\frac{1}{3Aw^2}\right) \\ \nonumber
			& \leq 4\cdot e^{c+\frac{1}{w^2}}\cdot C_r \cdot 2^{r+6+\frac{1}{24}} \pi^2  (3A)^{r+\frac{61}{24}} (r+2) \cdot w^{2r+5+\frac{1}{12}}.
\end{align*}
\end{proof}

\section{Proofs of Theorem \ref{thm:main_theorem_general_case} and \ref{thm:main_theorem_special case}}\label{sec:hermite-condition}
Here we prove Theorem \ref{thm:main_theorem_general_case} and \ref{thm:main_theorem_special case} using Theorem \ref{thm:formula-for-pl-in-w}. We start with a sufficient condition (due to Hermite) for establishing the hyperbolicity of a polynomial.
\subsection{Hankel determinant}
For a given real polynomial $P(x)=a_d x^d+a_{d-1} x^{d-1}+\cdots+ a_1x+a_0$, let $S_k:=\sum_{i=1}^d \lambda_i^k$ be the $k$th power sum of real roots. Then take
\begin{equation}
\Delta_m(P(x)):= \begin{vmatrix}
				S_0 & S_1 & \cdots & S_{m-1} \\
				S_1 & S_2 & \cdots & S_{m} \\
				S_2 & S_3 & \cdots & S_{m+1} \\
				\vdots & \vdots & \vdots & \vdots \\
				S_{m-1} & S_{m} & \cdots & S_{2m-2} 
			\end{vmatrix}	= \sum_{i_1<i_2<\cdots < i_m} \prod_{a<b} (\lambda_{i_a}-\lambda_{i_b})^2.
\end{equation}
For convenience, we define
\begin{equation}
D_{d,m}(P(x))=D_{d,m}(a_0,a_1,\cdots,a_d):=a_d^{2m-2}\cdot\Delta_m(P(x)),
\end{equation}
so that $D_{d,d}(a_0,a_1,\cdots,a_d)$ is the discriminant of $P(x)$, and $D_{d,m}(a_0,a_1,\cdots,a_d)$ is a homogeneous polynomial of degree $2m-2$ in the coefficients $a_i$. A theorem of Hermite \cite{obrechkoff2003zeros} says the hyperbolicity of $P(x)$ is implied by the positivity of $D_{d,m}(a_0,a_1,\cdots,a_d)$ for all $2\leq m\leq d$.

We will prove Theorems \ref{thm:main_theorem_general_case} and \ref{thm:main_theorem_special case} by showing that 
\begin{align}
	\D_{d,\pl,m}(n) & := D_{d,m}\left(\dfrac{J_{\pl}^{d,n}(x)}{\pl(n)}\right)\\ \nonumber
	 &=D_{d,m}\left(1,\binom{d}{1}\dfrac{\pl(n+1)}{\pl(n)},\binom{d}{2}\dfrac{\pl(n+2)}{\pl(n)},\cdots,\binom{d}{d}\dfrac{\pl(n+d)}{\pl(n)}\right)>0,
\end{align}
for each $m=2,\cdots,d$ and all $n$ greater than $N_{\pl}(d)$ mentioned in the theorem. Note that the limit $\lim_{n\to\infty}\frac{\pl(n+j)}{\pl(n)}=1$ for a fixed $j$, which implies that $\lim_{n\to\infty} J_{\pl}^{d,n}(x)/\pl(n)=(x+1)^d$. This implies that $\D_{d,\pl,m}$ approaches $0$ in the limit as $n\to \infty$, which a priori, makes the sign of $\D_{d,\pl,m}(n)$ difficult to determine.

However, we can determined the rate at which $\D_{d,\pl,m}(n)$ approaches $0$ and the coefficient of the leading term using the results in \cite{ono2022turan} and \cite{griffin2019}. More precisely, we have that 
\begin{equation}
	\lim_{n\to\infty}\dfrac{1}{\delta(n)^{m(m-1)}}\Delta_m\left(\dfrac{J_{\pl}^{d,n}(x)}{\pl(n)} \right)=\lim_{n\to\infty}\Delta_m\left(\dfrac{J_{\pl}^{d,m}\left(\delta(n)x-e^{-\sqrt{3A}w(n)}\right)}{\pl(n)}\right)=\Delta_m(H_d(x)).
\end{equation}
If we do a change of variable using \eqref{eq:formula-for-w-and-delta} and use $\D_{d,\pl,m}$ notation, then this translates to 
\begin{equation}\label{limiting-behaviour-of-D_{d,pl,m}}
	\lim_{w\to 0}\dfrac{1}{w^{2m(m-1)}}\D_{d,\pl,m}(n)=\left(\dfrac{\sqrt{3A}}{2}\right)^{m(m-1)}\Delta_m(H_d(x)).
\end{equation}
Since Hermite polynomials have distinct real roots, the right hand side is a positive constant. We will exploit this fact by using $(2m(m-1)+1)$-Taylor polynomial of $\D_{d,\pl,m}(n)$ around 0. The constant term in the left hand side will be a constant multiple of $\Delta_m(H_d(x))$. We will then find explicit bounds for the remaining terms that are tending to zero.

To do this, we need to study ratios of plane partitions.

\subsection{Approximation of ratios of plane partition}
In this subsection, we give approximation function for the ratios of the plane partitions with the error function. For each non-negative integers $r$ and $j$, we define approximation function for $\frac{\pl(n+j)}{\pl(n)}$ by
\begin{equation}
	R_{r}(j,w):=\dfrac{\widehat{\pl}_r\left(\frac{w}{\left(1+\frac{3^{3/2}\sqrt{A}}{2} jw^3\right)^{\frac{1}{3}}}\right)}{\widehat{\pl}_r(w)}\sim \dfrac{\pl(n+j)}{\pl(n)}.
\end{equation}
In order to state precisely how well $R_{r}(j,w)$ approximates $\pl(n+j)/\pl(n)$, let's define
\begin{align}\label{eq:upper-bound-on-L_r}
	L_r(w) :=\dfrac{E_r(w)}{\widehat{\pl}_r(w)} \leq \dfrac{\sqrt{3A}}{0.1485}\cdot C_r \cdot 2^{r+9+\frac{1}{24}} \pi^3  (3A)^{r+1}(r+2) \cdot w^{2r+3}.
\end{align}
Then we have the following lemma.
\begin{lemma}\label{lem:error-bound}
	For all $n\geq 1$, we have 
	\[
	\left| \dfrac{\pl(n+j)}{\pl(n)}-R_r(j,w) \right|\leq R_r(j,w)\left|\dfrac{2 L_r(w)}{1-L_r(w)}\right|.
	\]
\end{lemma}
\begin{proof}
	We have that $\widehat{E}_r(w)=\pl(n)-\widehat{\pl}_r(w)=E_r^{{\rm{maj}}}(n)+E^{{\rm{min}}}(n).$ By direct calculations, we have
	\begin{align*}
		\left|\dfrac{\pl(n+j)}{\pl(n)}-R_r(j,w) \right| &= \left|\dfrac{\pl(n+j)}{\pl(n)}-\dfrac{\widehat{\pl}_r\left(w(n+j)\right)}{\widehat{\pl}_r(w(n))}\right|\\
		&=\dfrac{\widehat{\pl}_r\left(w(n+j)\right)}{\widehat{\pl}_r(w(n))}\left|\dfrac{1+\frac{\widehat{E}_r(w(n+j))}{\widehat{\pl}_r(w(n+j))}}{1+\frac{\widehat{E}_r(w(n))}{\widehat{\pl}_r(w(n))}}-1\right|\\
		&=R_r(j,w)\left|\dfrac{\frac{\widehat{E}_r(w+j)}{\widehat{\pl}_r(w(n+j))}-\frac{\widehat{E}_r(w(n))}{\widehat{\pl}_r(w(n))}}{1+\frac{\widehat{E}_r(w(n))}{\widehat{\pl}_r(w(n))}}\right|	\leq R_r(j,w)\left| \dfrac{2 L_r(w)}{1-L_r(w)} \right|.
	\end{align*}
\end{proof}

To study the behavior of $\pl(n+j)/\pl(n)$ for large $n$, we want to study $R_r(j,w)$ near $w=0$. To this end, let $A_{r,s}(j,w)$ be a degree $s-1$ Taylor polynomial of $R_r(j,w)$. Applying Lemma \ref{lem:error-bound} and Taylor's Theorem, we immediately obtain the following.
\begin{lemma}\label{lem:key-lemma}
	Let $n\geq 1$ and $w\in[0,\varepsilon]$ for some $0<\varepsilon\leq \frac{2^{1/3}}{\sqrt{3}A^{1/6}}$. Then we have that
	\[
		\dfrac{\pl(n+j)}{\pl(n)}=A_{r,s}(j,w)+E_{r,s}(j,w)w^s,
	\]
	where
	\begin{equation}\label{eq:tail-of-taylor-series}
		\left|E_{r,s}(j,w)\right|\leq \dfrac{1}{s!}\cdot\sup_{x\in[0,\varepsilon]}\left|R^{(s)}_r(j,x)\right|+\sup_{x\in[0,\varepsilon]}\left| R_r(j,x)\dfrac{2L_r(x)}{x^s(1-L_r(x))} \right|.
	\end{equation}
\end{lemma}

In view of \eqref{eq:tail-of-taylor-series}, for each choice of $s$, and $2r+3\geq s$ we have that  $E_{r,s}(j,w)$ is bounded. From here onward we make the choice of 
\begin{equation}
	s=2d(d-1)+1 \text{ and } r=d(d-1).
\end{equation}
We also denote 
\begin{equation}
	\varepsilon := \varepsilon_{d(d-1),d}.
\end{equation}
It is easy to see that 
\begin{equation}
	0\leq \left|L_r(w)\right| <\frac{1}{2},\quad w\in[0,\varepsilon],
\end{equation}
so we get that 
\begin{equation}\label{eq:error-bound-on-ratio}
	\left| \dfrac{\pl(n+j)}{\pl(n)}-R_r(j,w) \right|\leq R_r(j,w)\left| 4L_r(w) \right|.
\end{equation}

To use Lemma \ref{lem:key-lemma} effectively, we need to obtain a bound on derivatives of $R_r(j,w)$.
\subsection{Bound on nth derivative of $R_r(j,w)$}
The polynomial $\D_{d,\pl,m}(n)$ is homogeneous of degree $2m-2$ in the coefficinets of $J^{d,n}_{\pl}(X)/\pl(n)$ and homogeneous of degree $m(m-1)$ in its roots. So, it has the form
\begin{equation}\label{eq:form-of-D_d,pl,m}
	\D_{d,\pl,m}(n)=\sum_{i_1+i_2+\cdots +i_{2m-2}=m(m-1)}A_{i_1,i_2,\cdots,i_{2m-2}}\cdot\prod_{k=1}^{2m-2} \binom{d}{i_k}\dfrac{\pl(n+d-i_k)}{\pl(n)}
\end{equation}
for some constants $A_{i_1,i_2,\cdots,i_{2m-2}}$. To bound the errors when we expand in terms of $w$, we find bounds on the derivatives $R_r^{(s)}(j, w)$ for $w$ in the interval $[0, \varepsilon]$. For convenience, let $t = t(j) := \frac{3^{3/2}\sqrt{A}}{2} j$.
\begin{lemma}\label{lem:nth-derivative-of-R_r}
	Assume that $w\in [0,\varepsilon]$ with $\varepsilon$ as above. Then for each $n\geq 1$, we have that
	\begin{align*}
	\left| R_r^{(n)}(j,w) \right|
	& \leq  n! \binom{n+3}{3} e^{g(w)} e^{(3tw^2+3tw)n}\cdot t^{\frac{7n}{3}}\cdot \left((r+1)^2\cdot 6213\cdot\alpha_{\frac{n}{2}}\Gamma\left(n+\dfrac{13}{12}+1\right)\right)^{n},
	\end{align*}
	where
	$
		g(w)=\frac{(1+tw^3)^{2/3}-1}{w^2}.
	$
\end{lemma}
\begin{proof}
	The idea of the proof is to use the product rule to split up $R_r(j, w)$ into four more manageable parts and use Faà di Bruno’s formula for iterated applications of the chain rule to evaluate each part as needed. This formula says that for differentiable functions $f(x)$ and $g(x)$, we have
	\begin{equation}\label{eq:Faa-di-Bruno}
		\dfrac{d^n}{dx^n}f(g(x))=\sum_{m_1+ 2m_2+\cdots+nm_n=n}\dfrac{n!}{m_1!m_2!\cdots m_n!}f^{(m_1+m_2+\cdots+m_n)}(g(x))\prod_{j=1}^{n}\left( \dfrac{g^{(j)}(x)}{j!} \right)^{m_j}.
	\end{equation}
	First we define 
	\begin{align*}
		A:= & A(t,w) =e^{\frac{t^2w^4+2tw}{\left(\left(1+tw^3\right)^{2/3}+\frac{1}{2}\right)^2+\frac{3}{4}}}, \\
		 B:= & B(t,w) =\dfrac{1}{(1+tw^3)^{25/36}}, \\
		C:= & \Tilde{C}_r(t,w) =\sum_{s=0}^{r+1}\sum_{m=0}^{r+1}f_{s,m} \left(\dfrac{w}{(1+tw^3)^{1/3}}\right)^{2s+2m}, \\
		 D:= & D(t,w) =\dfrac{1}{\sum_{s=0}^{r+1}\sum_{m=0}^{r+1}f_{s,m} w^{2s+2m}}.
	\end{align*}
	Then we have that 
	\begin{equation}\label{eq:general-form-of-nth-derivative-of-R_r}
		R_r^{(m)}(j,w)= \sum_{m_1+m_2+m_3+m_4=m}\dfrac{m!}{m_1!m_2!m_3!m_4!}\left(\dfrac{d^{m_1}A}{dw^{m_1}}\right)\left(\dfrac{d^{m_2}B}{dw^{m_2}}\right)\left(\dfrac{d^{m_3}C}{dw^{m_3}}\right)\left(\dfrac{d^{m_4}D}{dw^{m_4}}\right).
	\end{equation}
	We will prove the bound on $n$-th derivative of $A$ in full detail, and others will follow similarly. We use \eqref{eq:Faa-di-Bruno} with $f(x)=e^x$ and $g(w)=\frac{t^2w^4+2tw}{\left(\left(1+tw^3\right)^{2/3}+\frac{1}{2}\right)^2+\frac{3}{4}}$, and find that
	\begin{equation}\label{eq:derivative-formula-for-A}
		\dfrac{d^nA}{dw^n}=\dfrac{d^n}{dw^n}f(g(w))=\sum_{m_1+2m_2+\cdots + n m_{n}=n} \dfrac{n!}{m_1!m_2!\cdots m_n!} e^{g(w)}\prod_{k=1}^{n}\left(\dfrac{g^{(k)}(w)}{k!}\right)^{m_k}.
	\end{equation}
	We write $g(w)=g_1(w)\cdot g_2(w)$, where $g_1(w)=t^2w^4+2tw$, and $g_2(w)=\frac{1}{\left(\left(1+tw^3\right)^{2/3}+\frac{1}{2}\right)^2+\frac{3}{4}}$. We find that
	\begin{equation}
		\left|g_1^{(n)}(w)\right|\leq 4! t^2,
	\end{equation}
	and again using \eqref{eq:Faa-di-Bruno}, we obtain
	\begin{align}
		\left| g_2^{(n)}(w) \right| \leq & \sum_{m_1+2m_2+\cdots+nm_n=n}\dfrac{n!}{m_1!m_2!\cdots m_n!}\dfrac{\left(\sum m_i\right)!}{\left(\left(\left(1+tw^3\right)^{2/3}+\frac{1}{2}\right)^2+\frac{3}{4}\right)^{\sum m_i}}\\ \nonumber
		 & \hspace{3cm} \times\prod_{k=1}^n \left( 2^{k+1} \left(\left(1+tw^3\right)^{\frac{2}{3}}+\frac{1}{2} \right) e^{(3tw^2+3tw)k} t^{\frac{k}{3}} \right)^{m_k}\\ \nonumber
		 \leq & \sum_{m_1+2m_2+\cdots+nm_n=n}\dfrac{n!}{m_1!m_2!\cdots m_n!}\dfrac{\left(\sum m_i\right)!}{\left(\left(\left(1+tw^3\right)^{2/3}+\frac{1}{2}\right)^2+\frac{3}{4}\right)^{\sum m_i}}\\ \nonumber
		 & \hspace{3cm} \times 2^{{\small \sum i m_i+m_i}} \left(\left(1+tw^3\right)^{\frac{2}{3}}+\frac{1}{2} \right)^{\sum m_i} e^{(3tw^2+3tw)\sum i m_i}\cdot t^{\frac{\sum i m_i}{3}}\\ \nonumber
		\leq & n! 2^{2n} e^{(3tw^2+3tw)n} t^{\frac{n}{3}} \cdot \sum_{m_1+2m_2+\cdots+nm_n=n}\dfrac{\left(\sum m_i\right)}{m_1!m_2!\cdots m_n!}\\ \nonumber
		\leq & n! 2^{3n} e^{(3tw^2+3tw)n} t^{\frac{n}{3}},
	\end{align}
	where we use the fact that the sum $\sum_{m_1+2m_2+\cdots+nm_n=n}1$ is  counting the number of ordered partitions of $n$. This gives us that
	\begin{align}\label{eq:bound-on-g}
		\left| g^{(n)}(w)\right| \leq & \sum_{m_1+m_2=n}\dfrac{n!}{m_1!m_2!}\cdot\dfrac{d^{m_1}g_1}{dw^{m_1}}\cdot \dfrac{d^{m_2}g_2}{dw^{m_2}} \\ \nonumber
		\leq & \sum_{m_1+m_2=n}\dfrac{n!}{m_1!m_2!} (4! t^2) m_2! 2^{3m_2} e^{(3tw^2+3tw)m_2} t^{\frac{m_2}{3}}\\ \nonumber
		\leq & 4! n! 2^{3m} t^{2+\frac{n}{3}}e^{(3tw^2+3tw)n}\sum_{\substack{m_1+m_2=n \\ m_1\leq 4}}\dfrac{1}{m_1!} \\ \nonumber
		\leq & n! 2^{3n} 41 \cdot t^{2+\frac{n}{3}} e^{(3tw^2+3tw)n}.
	\end{align}
	So, combining \eqref{eq:derivative-formula-for-A}$-$\eqref{eq:bound-on-g}, we obtain
	\begin{align}
		\left| \dfrac{d^{n}A}{dw^{n}} \right| \leq &  n! 2^{3n} (41)^n e^{g(w)}\cdot t^{2n+\frac{n}{3}} e^{(3tw^2+3tw)n} \sum_{m_1+2m_2+\cdots+nm_n=n}\dfrac{(m_1+m_2\cdots m_n)!}{m_1!m_2!\cdots m_n} \\		\nonumber
											\leq & e^{g(w)}n! e^{(3tw^2+3tw)n} 2^{4n} (41)^n t^{\frac{7}{3}n}.
	\end{align}
	Next, it can be shown that 
	\begin{equation}
		\left| \dfrac{d^n B}{dw^n} \right| \leq n! e^{(3tw^2+3tw)n} t^{\frac{n}{3}}.
	\end{equation}
	Now, we give bounds on $C$ and	$D$. First, we let $W(w)=\frac{w}{(1+tw^3)^{1/3}},$ then we have 
	\begin{align}
		\left| W^{(n)}(w) \right| \leq & w\dfrac{d^n}{dw^n}\left(\dfrac{1}{(1+tw^3)^{1/3}}\right) + n\dfrac{d^{n-1}}{dw^{n-1}}\left(\dfrac{1}{(1+tw^3)^{1/3}}\right).
	\end{align}
	Direct calculation gives
	\begin{align*}
		\left|\dfrac{d^n}{dw^n} \left(\dfrac{1}{(1+tw^3)^{1/3}}\right) \right| \leq & \left|\sum_{m_1+2m_2+3m_3=n}\dfrac{n!}{m_1!m_2!m_3!}\binom{-1/3}{m_1+m_2+m_3}\dfrac{(3tw^2)^{m_2}(3tw)^{m_2}(t)^{m_3}}{(1+tw^3)^{\frac{1}{3}+m_1+m_2+m_3}}\right| \\
		\leq & n! e^{3tw^2+3tw} t^{\frac{n}{3}}.
	\end{align*}
	Since $|w|<1$, this implies that
	\begin{equation}
		\left| W^{(n)}(w) \right|\leq 2 n! e^{3tw^2+3tw} t^{\frac{n}{3}}.
	\end{equation}
	The definition of $C$ the gives
	\begin{align}
		\left| \dfrac{d^n C}{dw^n} \right| \leq & \Bigg|\sum_{\substack{s=0\\ 2s+2m\geq n}}^{r+1}\sum_{m=0}^{r+1}  f_{s,m}  \dfrac{(2s+2m)!}{n!} \\ \nonumber
		& \times \sum_{m_1+2m_2+\cdots+nm_n=n}\dfrac{n!}{m_1!m_2!\cdots m_n!}W^{2s+2m-n} \prod_{k=1}^n\left( 2 e^{3tw^2+3tw} t^{\frac{k}{3}} \right)^{m_k}\Bigg|\\ \nonumber
		\leq & 2^{2n} n! e^{(3tw^2+3tw)n} t^{\frac{n}{3}} (r+1)^2 \sum_{2s+2m=n}\left| f_{s,m} \right|.
	\end{align}
	Here in the second inequality, we use that for $w\in[0,\varepsilon]$, the sum on right hand side is dominated by constant term. Now we look into $C$ and $D$, which we need to bound $f_{s,m}$. First notice that $|\beta_s|\leq 1$ for $s\leq 54$, and for $s\geq 55$, we have
	\begin{align}
		\left|\beta_{s}\right|  & = \left|\dfrac{1}{s!} \sum_{m_1+2m_2+\cdots sm_s=s}\dfrac{s!}{m_1!m_2!\cdots m_s!}\prod_{k=1}^{s}(-\alpha_k)^{m_k} \right| \\ \nonumber
								& = \left| -\alpha_s + \sum_{m_1+2m_2+\cdots (s-1)m_{s-1}=s}\dfrac{s!}{m_1!m_2!\cdots m_s!}\prod_{k=1}^{s-1}(-\alpha_k)^{m_k} \right|
								 \leq \alpha^s,
	\end{align}
	since there is an alternating signs in each terms and $\alpha_s$ dominates. Also, we have 
	\begin{align}
		|b_{s,m}| = & \left|\dfrac{1}{(2m)!}\dfrac{d^{2m}\,\,}{dy^{2m}}\left(\dfrac{(1+y)^{2s+2m+\frac{13}{12}}}{(3+2y)^{\frac{1}{2}}}\right)\Bigg|_{y=0}\right| \\ \nonumber
		\leq & \dfrac{1}{(2m)!}\sum_{m_1+m_2=2m}\dfrac{(m_1+m_2)!}{m_1!m_2!}\Gamma\left(2s+2m-m_1+\frac{13}{12}+1\right) \Gamma\left(-\frac{1}{2}-m_2+1\right)\left(\dfrac{3}{2}\right)^{m_2} \sqrt{3} \\ \nonumber
		\leq & \dfrac{1}{(2m)!}\Gamma\left(2s+2m+\frac{13}{12}+1\right) \Gamma\left(\frac{1}{2}\right)\sqrt{3} \sum_{m_1+m_2=2m}\dfrac{(2m)!}{m_1!m_2!}\left(\frac{3}{2}\right)^{m_2} \\ \nonumber
		\leq &\dfrac{1}{(2m)!} \sqrt{3}\cdot \Gamma\left(2s+2m+\frac{13}{12}+1\right)\cdot \Gamma\left(\frac{1}{2}\right) \cdot \left( 1+\dfrac{3}{2} \right)^{2m} \\ \nonumber
	\leq & \dfrac{1}{(2m)!} \sqrt{3} \cdot \dfrac{5^{2m}}{2^{2m}} \cdot \Gamma\left(2s+2m+\frac{13}{12}+1\right)\cdot \Gamma\left(\frac{1}{2}\right).
	\end{align}
	So, we get that 
	\begin{align}
		\sum_{2s+2m=n}\left| f_{s,m} \right| & \leq 5\cdot 3^{\frac{n}{2}+\frac{25}{24}+\frac{1}{2}}\cdot\Gamma\left( n+\frac{13}{12}+1 \right) \sqrt{\pi} A^{\frac{13}{24}}\alpha_{\frac{n}{2}}^n \left(\dfrac{25}{4}\right)^{n/2}.
	\end{align}
	Combining these facts we obtain
	\begin{align}
		\left| \dfrac{d^n C}{dw^n} \right| & \leq 2^{2n} n! e^{(3tw^2+3tw)n} t^{\frac{n}{3}} (r+1)^2 \left( \dfrac{75}{A}\right)^{\frac{n}{2}} 3^{\frac{37}{24}} 5\Gamma\left(n+\frac{13}{12}+1\right)\sqrt{\pi} A^{\frac{13}{12}}\alpha_{\frac{n}{2}}^n.
	\end{align}
	Using a similar argument and the fact that $|D|\leq \frac{1}{0.1485}$ when $w\in [0,\varepsilon]$ using Mathematica , we get that
	\begin{equation}
		\left| \dfrac{d^n D}{dw^n} \right| \leq \dfrac{n!}{0.1485} \left(\dfrac{(r+1)^2\cdot \sqrt{75}\cdot \sqrt{\pi} 3^{\frac{37}{24}}\cdot 5\cdot A^{\frac{13}{24}}\cdot 2\alpha_{\frac{n}{2}}}{0.1485}\right)^n\Gamma\left(n+\dfrac{13}{12}+1\right)^n.
	\end{equation}
	Therefore, thanks to \eqref{eq:general-form-of-nth-derivative-of-R_r} we obtain
	\begin{align}
		\left| R_r^{(n)}(j,w)\right|\leq & \sum_{m_1+\cdots+m_4=n}\dfrac{m!}{m_1!\cdots m_4!} e^{g(w)}m_1! e^{(3tw^2+3tw)m_1} 2^{4m_1} (41)^{m_1} t^{\frac{7}{3}m_1}\cdot m_2! \\ \nonumber
		& \times e^{(3tw^2+3tw)m_2} t^{\frac{m_2}{3}} 2^{2m_3} m_3! e^{(3tw^2+3tw)m_3} t^{\frac{m_3}{3}} (r+1)^2 \left( \dfrac{75}{A}\right)^{\frac{m_3}{2}} 3^{\frac{37}{24}} \\ \nonumber
		&\times 5 \cdot \Gamma\left(m_3+\frac{13}{12}+1\right)\sqrt{\pi} A^{\frac{13}{12}}\alpha_{\frac{m_3}{2}}^{m_3} \dfrac{m_4!}{0.1485} \cdot\Gamma\left(m_4+\dfrac{13}{12}+1\right)^{m_4}  \\ \nonumber
		&   \times \left(\dfrac{(r+1)^2\cdot \sqrt{75}\cdot \sqrt{\pi} 3^{\frac{37}{24}}\cdot 5\cdot A^{\frac{13}{24}}\cdot 2\alpha_{\frac{m_4}{2}}}{0.1485}\right)^{m_4}\\ \nonumber
		\leq & n! e^{g(w)} e^{(3tw^2+3tw)n}\cdot t^{\frac{7n}{3}}\cdot \left(\dfrac{(r+1)^2\cdot \sqrt{75}\cdot \sqrt{\pi} 3^{\frac{37}{24}}\cdot 5\cdot A^{\frac{13}{24}}\cdot 2\alpha_{\frac{n}{2}}}{0.1485}\right)^{n} \\ \nonumber
		& \times \Gamma\left(n+\dfrac{13}{12}+1\right)^{n} \cdot \sum_{m_1+\cdots+m_4=n}1 \\ \nonumber
		\leq  n! \binom{n+3}{3} e^{g(w)}& e^{(3tw^2+3tw)n}\cdot t^{\frac{7n}{3}}\cdot \left(\dfrac{(r+1)^2\cdot \sqrt{75}\cdot \sqrt{\pi} 3^{\frac{37}{24}}\cdot 5\cdot A^{\frac{13}{24}}\cdot 2\alpha_{\frac{n}{2}}}{0.1485}\Gamma\left(n+\dfrac{13}{12}+1\right)\right)^{n} \\ \nonumber
		\leq  n! \binom{n+3}{3} e^{g(w)}& e^{(3tw^2+3tw)n}\cdot t^{\frac{7n}{3}}\cdot \left((r+1)^2\cdot 6213\cdot\alpha_{\frac{n}{2}}\cdot\Gamma\left(n+\dfrac{13}{12}+1\right)\right)^{n}.
	\end{align}
\end{proof}
Thanks to \eqref{eq:form-of-D_d,pl,m}, we want to estimate products of ratios of plane partition function values. Given $\uu{i}=(i_1,i_2,\cdots,i_{2m-2})$ with $i_1+i_2+\cdots+i_{2m-2}=m(m-1)$, let $T_{d,\pl,m}(\uu{i};w)$ be the degree $2m(m-1)$ Taylor polynomial of $\prod_{k=1}^{2m-2}R_r(d-i_k,w)$.
\begin{lemma}\label{lem:taylor-exp-of-ratios-of-planepartition}
	If  $w\in [0,\varepsilon]$, then we have that 
	\[
		\prod_{k=1}^{2m-2} \dfrac{\pl(n+d-i_k)}{\pl(n)}=T_{d,\pl,m}(\uu{i};w)+E_{d,\pl,m}(\uu{i};w) w^{2m(m-1)+1},
	\]
	where
	\begin{align*}
		\left| E_{d,\pl,m}(\uu{i};w) \right| & 
		   \leq    2\cdot \left( e^{\frac{\Gamma(2d^2)}{(2\pi)^{2d^2+2}}}6^{2d-2} (6A)^r (r+1)(2d-2) \right)^{\frac{12}{13}} \left(0.1485\cdot 2^{12} (3A)^{3}\pi^3 \right)^{\frac{12}{13}}.
	\end{align*}
\end{lemma}
\begin{proof}
By Lemma \ref{lem:key-lemma}, we can write that
	\[
		\prod_{k=1}^{2m-2} \dfrac{\pl(n+d-i_k)}{\pl(n)} = \prod_{k=1}^{2m-2}R_r(d-i_k,w)(1+U_{r,k}(w)) = \prod_{k=1}^{2m-2}R_r(d-i_k,w) +U_r(w),
	\]
	where
	\begin{align*}
		\left|U_r(w)\right| \leq &\prod_{k=1}^{2m-2}R_r(d-i_k,w)\left( \left(1+\dfrac{2L_r(w)}{1-L_r(w)}\right)^{2m-2}-1 \right) \\
		\leq & 2^{2m-2}\cdot(2m-2)\cdot 3^{2m-2} \cdot \left| \dfrac{2L_r(w)}{1-L_r(w)} \right| 
		\leq 2^{2m-2}\cdot(2m-2)\cdot 3^{2m-2} \left|4\cdot L_r(w)\right| \\
		\leq & 2^{2m-2}\cdot(2m-2)\cdot 3^{2m-2} 16\dfrac{\sqrt{3A}}{0.1485}\cdot C_r\cdot 2^{r+7+\frac{1}{24}} \pi^3  (3A)^{r+\frac{49}{24}} (r+2) w^{2r+3}.
	\end{align*}
	Here we use that $\left| R_r(d-i_k,w) \right| \leq 2$ throughout and \eqref{eq:upper-bound-on-L_r} in the last inequality. If we choose $s=2m(m-1)+1$, then we have that
	\begin{align}
		\left|\dfrac{U_r(w)}{w^s}\right| &\leq 2^{2m-2}\cdot(2m-2)\cdot 3^{2m-2} 16\dfrac{\sqrt{3A}}{0.1485}\cdot C_r\cdot 2^{r+7+\frac{1}{24}} \pi^3  (3A)^{r+\frac{49}{24}} (r+2) \cdot \varepsilon^{2}\\ \nonumber
		& \leq  6^{2d-2}\cdot(2d-2)\cdot \dfrac{\sqrt{3A}}{0.1485}\cdot C_r\cdot 2^{r+11+\frac{1}{24}} \pi^3  (3A)^{r+\frac{49}{24}} (r+2)  \\ \nonumber
		&  \times \left( e^{\frac{\Gamma(2d^2)}{(2\pi)^{2d^2+2}}}6^{2d-2} (6A)^r (r+1) \right)^{-\frac{2}{3}} \left(7 2^{12} (3A)^{3} \right)^{-\frac{2}{3}}  \\ \nonumber
		& \leq  \left( e^{\frac{\Gamma(2d^2)}{(2\pi)^{2d^2+2}}}6^{2d-2} (6A)^r (r+1)(2d-2) \right)^{\frac{1}{3}} \left(0.1485\cdot 2^{12} (3A)^{3}\pi^3 \right)^{\frac{1}{3}} ,
	\end{align}
	where we get the last inequality from \eqref{eq:C_r}.  On the other hand, using the product rule and Lemma \ref{lem:nth-derivative-of-R_r} we obtain
	\begin{align}
		\dfrac{1}{s!}\Bigg| \dfrac{d^s}{dw^s}&\prod_{k=1}^{2m-2} R_r(d-i_k,w) \Bigg| \\ \nonumber
		 & \leq  e^{(2m-2)g(\varepsilon)}\left(e^{3t\varepsilon+3t\varepsilon^2}t^{\frac{7}{3}}(r+1)^2\cdot 6213\cdot \alpha_{\frac{s}{2}}\Gamma\left(s+\dfrac{13}{12}+1\right)\right)^s \\ \nonumber
		& \times \sum_{n_1+n_2+\cdots+n_{2m-2}=m(m-1)}\binom{n_1+3}{3}\binom{n_2+3}{3}\cdots\binom{n_{2m-2}+3}{3}.
	\end{align}
	The largest term in the sum on the right hand side occurs if each $n_i$ is equal, which in turn is bounded by replacing each $n_i$ with $m\geq \frac{m(m-1)}{2m-2}$. Counting the number of terms, we see that the sum is bounded above by
	\begin{align*}
		\sum_{n_1+n_2+\cdots+n_{2m-2}=2m(m-1)}\binom{n_1+3}{3}\binom{n_2+3}{3}& \cdots  \binom{n_{2m-2}+3}{3}  \\
		& \leq \binom{m+4}{3}^{2m-2} \binom{2(m-1)(m+1)}{2m-3} \\
		& \leq \left(\dfrac{5}{2}m^3\right)^{2m-2}(2m^2)^{2m-2} = (5m^5)^{2m-2}.
	\end{align*}
	This shows that
	\begin{align*}
		\Bigg| \dfrac{d^s}{dw^s}\prod_{k=1}^{2m-2} R_r(d-i_k,w) & -T_{d,\pl,m}(\uu{i};w)\Bigg| \\
		 & \leq   e^{(2m-2)g(\varepsilon)}\left(e^{3t\varepsilon+3t\varepsilon^2}t^{\frac{7}{3}}(r+1)^2\cdot 6213\cdot \alpha_{\frac{s}{2}}\right)^s \cdot \\
		 & \qquad\qquad\qquad \times\Gamma\left(s+\dfrac{13}{12}+1\right)^s\cdot(5m^5)^{2m-2}\cdot w^s + \\
		 & + \left( e^{\frac{\Gamma(2d^2)}{(2\pi)^{2d^2+2}}}6^{2d-2} (6A)^r (r+1)(2d-2) \right)^{\frac{1}{3}} \left(0.1485\cdot 2^{12} (3A)^{3}\pi^3 \right)^{\frac{1}{3}}w^s \\
		 & \leq  2\cdot \left( e^{\frac{\Gamma(2d^2)}{(2\pi)^{2d^2+2}}}6^{2d-2} (6A)^r (r+1)(2d-2) \right)^{\frac{1}{3}} \left(0.1485\cdot 2^{12} (3A)^{3}\pi^3 \right)^{\frac{1}{3}} w^s,
	\end{align*}
	where the above inequality follows by noticing that the second part of the sum is larger of the two. This is true since the second function has exponential growth rate and the first one has polynomial growth in the factorial, so we just need to check when second part becomes bigger than first one, which happens when $d\geq 4$.
\end{proof}
In order to finish bounding the monomials in \eqref{eq:form-of-D_d,pl,m}, we need the following result proved in a similar way as \cite[Lemma 4.3]{larson2019}.

\begin{lemma}\label{lem:bound-on-binomial}
	Suppose $0\leq m\leq d$ and $i_1+i_2+\cdots+i_{2m-2}=m(m-1)$ for positive integers $i_k$. Then we have that
	\[
		\left|\left(\dfrac{2}{\sqrt{3A}}\right)^{m(m-1)}\prod_{k=1}^{2m-2}\binom{d}{i_k}\right|\leq \left(\frac{4e}{\sqrt{3A}}\right)^{d(d-1)}.
	\]
\end{lemma}
\begin{proof}
	The product $\prod_{k=1}^{2m-2}\binom{d}{i_k}$ is maximized when all of $i_k$ are equal and equal to $\frac{m}{2}$. Using standard bounds on binomial coefficients, we therefore have that 
	\[
	\left(\dfrac{2}{\sqrt{3A}}\right)^{m(m-1)}\prod_{k=1}^{2m-2}\binom{d}{i_k}\leq \left(\frac{4ed}{\sqrt{3A}m}\right)^{m(m-1)} \leq \left(\frac{4e}{\sqrt{3A}}\right)^{d(d-1)},
	\]
	since the right hand side is maximized when $m=d$.
\end{proof}

We need one more lemma which gives the necessary bounds on the coefficients $A_{i_1,i_2,\cdots,i_{2m-2}}$ to achieve the required bound on $\D_{d,\pl,m}(n)$.
\begin{lemma}[\cite{larson2019}, Lemma 4.4]\label{lem:bound-on-A-i}
	If $n> \frac{2}{\sqrt{3A}\varepsilon^3}$ and $A_{i_1,i_2,\cdots,i_{2m-2}}$ is as in \eqref{eq:form-of-D_d,pl,m}, then we have 
	\[
		\sum_{i_1,i_2,\cdots,i_{2m-2}}|A_{i_1,i_2,\cdots,i_{2m-2}}|\leq d^{2d}\cdot 2^{d(d-1)}.
	\]
\end{lemma}

Because the limiting behavior of $J_{\pl}^{d,n}(x)$ is modeled by Hermite polynomials, we need the following lemma.
\begin{lemma}[\cite{larson2019}, Lemma 4.5]\label{lem:hermit-polynomial} 
	For each $m\leq d$, we have that $\Delta_m(H_d(x))\geq 1$.
\end{lemma}

\subsection{Proof of Theorem \ref{thm:main_theorem_general_case}}\label{sec:proof-of-general-case}
Suppose that $n\geq \frac{2}{\varepsilon\sqrt{27A}}$ so that $w\in[0,\varepsilon]$. By \eqref{eq:form-of-D_d,pl,m}, we have that
\begin{align*}
	\dfrac{\D_{d,\pl,m}(n)}{w^{2m(m-1)}}= & \sum_{i_1+\cdots + i_{2m-2}=2m(m-1)} \dfrac{A_{i_1,\cdots, i_{2m-2}}}{w^{2m(m-1)}}\cdot \prod_{k=1}^{2m-2}\binom{d}{i_k}\left( T_{d,\pl,m}(\uu{i};w)+E_{d,\pl,m}(\uu{i};w)w^{2m(m-1)+1} \right) \\
	= & \left(\dfrac{\sqrt{3A}}{2}\right)^{m(m-1)} \Delta_{m}(H_d(x))+w\cdot \mathcal{E}_{d,\pl,m}(w),
\end{align*}
where by Lemma \ref{lem:taylor-exp-of-ratios-of-planepartition}, \ref{lem:bound-on-binomial}, \ref{lem:bound-on-A-i} and the choice of $\varepsilon$, we have that
\begin{align*}
	\left(\dfrac{2}{\sqrt{3A}}\right)^{m(m-1)}\cdot & \left| \mathcal{E}_{d,\pl,m}(w)\right| \cdot w \leq d^{2d} 2^{d(d-1)}\cdot \left(\frac{4e}{\sqrt{3A}}\right)^{d(d-1)} \cdot \\
	&  2\cdot \left( e^{\frac{\Gamma(2d^2)}{(2\pi)^{2d^2+2}}}6^{2d-2} (6A)^r (r+1)(2d-2) \right)^{\frac{1}{3}} \left(0.1485\cdot 2^{12} (3A)^{3}\pi^3 \right)^{\frac{1}{3}}\cdot \varepsilon <1.
\end{align*}
Since $\Delta_{m}(H_d(x))\geq 1$, it follows that $\D_{d,\pl,m}(n)>0$ and therefore $J_{\pl}^{d,m}(x)$ is hyperbolic. We use \eqref{eq:formula-for-w-and-delta} to get the upper bound on $N_{\pl}(d)$.
\hfill $\square$

\subsection{Proof of Theorem \ref{thm:main_theorem_special case}}\label{sec:proof-of-special-case}
We now prove Theorem \ref{thm:main_theorem_special case} by bounding  the error terms that accumulate from approximating $\pl(n+j)/\pl(n)$ by the $(s-1)$th Taylor polynomial $A_{r,s}(j,w)$ of $R_r(j,w)$ using Lemma \ref{lem:key-lemma}, in the polynomial expression for $\D_{d,\pl,m}(n(w))$. This gives us that there exists an $\varepsilon$ such that $\D_{d,\pl,m}(n(w))\geq 0$ for all $w\in[0,\varepsilon]$ (i.e. $n\geq n_{\varepsilon}$)  which in turn allows us to reduce to checking finitely many cases.

Using the Newton-Girard identities to write the power sum of the roots in terms of elementary symmetric function, one can generate symbolic expressions for the polynomials $D_{d,m}(a_0,a_1,\cdots,a_d)$ in terms of $a_0,a_1,\cdots,a_d$. To obtain $\D_{d,\pl,m}(n)$, we substitute
\[
	\binom{d}{j}\left( A_{r,s}(j,w)+E_jw^s \right)
\]
in for $a_j$ in these polynomials, introducing $E_j$ as a variable.
For example, when $d=3, r=5$ and $s=10$, we have that
\[
	D_{3,2}(a_0,a_1,a_2,a_3) = 2 a_2^2-6 a_1 a_3.
\]
So we get that
\[
	\D_{3,\pl,2}(n)=18 \left(A_{5,10}(w)+E_2 w^{10}\right)^2-18 \left(A_{5,10}(w)+E_1 w^{10}\right) \left(A_{5,10}(w)+E_3 w^{10}\right).
\]
This gives rise to a polynomial expression in $w$ whose coefficients are polynomials in $E_j$. It turns out that all the coefficients less than $k=2m(m-1)$ vanish in the expression. So diving by $w^k$, gives an expression of the form
\begin{equation}\label{eq:special-case-poly-for-D-d,pl,m}
	\D_{d,\pl,m}(w)=c_0+c_1w+c_2(E_1,E_2,\cdots,E_d)w^2+\cdots + c_{(2m-2)s-k}(E_1,E_2,\cdots,E_d) w^{(2m-2)s-k},
\end{equation}
for each $2\leq m\leq d$, where $c_0$ and $c_1$ are positive constants.

We use Mathematica \cite{mathematica} to calculate the upper bound on $E_j=E_{r,s}(j,w)$ for $w\in [0,\varepsilon]$ using Lemma \ref{lem:key-lemma}, where we choose 
\begin{align*}
	& r=5,7,10,10,10 ,\, s=10,12,18,18,20 \text{ and }	\varepsilon = 0.051, 0.032, 0.06, 0.03,0.02\\ & \text{ for $d=3,4,5,6,7$ respectively.}
\end{align*}
With the help of Mathematica again, we minimize \eqref{eq:special-case-poly-for-D-d,pl,m} using these bounds and it turns out that in each case the minimum is positive for all $2\leq m\leq d$, which proves the hyperbolicity of $J_{\pl}^{d,n}(x)$ for $d=3,4,5,6$, and $7$ for all $n\geq n_{\varepsilon}$.

To get this $n_{\varepsilon}$, we use the condition that $w\leq \varepsilon$ and the relation between $w$ and $n$ given by \eqref{eq:formula-for-w-and-delta}. This gives us that $N_{\pl}(3)\leq 2647$, $N_{\pl}(4)\leq 10714, N_{\pl}(5)\leq 1626, N_{\pl}(6)\leq 13003$ and $N_{\pl}(7)\leq 43883$. Checking the finite number of remaining possible counter examples directly now proves the theorem. Annotated Mathematica code to implement the full procedure described above is at \cite{mathematicacode}.

\bibliographystyle{unsrt}
\bibliography{main}

\end{document}